\documentclass{amsart}

\usepackage{amssymb,amscd,amsthm,amsxtra}
\usepackage{dsfont,latexsym}
\usepackage{amsrefs}
\usepackage{mathrsfs}

\newtheorem{theorem}{Theorem}[section]
\newtheorem{corollary}[theorem]{Corollary}
\newtheorem{lemma}[theorem]{Lemma}

\newcommand{\R}{{\mathds R}}
\newcommand{\N}{{\mathds N}}

\renewcommand{\le}{\leqslant}

\renewcommand{\ge}{\geqslant}

\begin{document}

\author[S. Dipierro]{Serena Dipierro}
\address[Serena Dipierro]{Maxwell Institute for Mathematical
Sciences and School of Mathematics,
University of Edinburgh, James Clerk Maxwell Building,
Peter Guthrie Tait Road, Edinburgh EH9 3FD, United Kingdom}
\email{serena.dipierro@ed.ac.uk}

\author[O. Savin]{Ovidiu Savin}
\address[Ovidiu Savin]{Department of Mathematics, Columbia University,
2990 Broadway,
New York NY 10027, USA}
\email{savin@math.columbia.edu}

\author[E. Valdinoci]{Enrico Valdinoci}
\address[Enrico Valdinoci]{Weierstra{\ss} Institut
f\"ur Angewandte Analysis und Stochastik,
Mohrenstra{\ss}e 39, 10117 Berlin, Germany}
\email{enrico.valdinoci@wias-berlin.de}

\thanks{
The first author has been supported by EPSRC grant  EP/K024566/1
``Monotonicity formula methods for nonlinear PDEs''.
The second author has been supported by 
NSF grant DMS-1200701.
The third author has been supported by ERC grant 277749 ``EPSILON Elliptic
Pde's and Symmetry of Interfaces and Layers for Odd Nonlinearities''
and PRIN grant 201274FYK7
``Aspetti variazionali e
perturbativi nei problemi differenziali nonlineari''.}

\title[$s$-harmonic functions are dense]{All functions
are locally $s$-harmonic\\
up to a small error}

\begin{abstract}
We show that we can approximate every
function~$f\in C^{k}(\overline{B_1})$
with a $s$-harmonic function in~$B_1$ that vanishes outside
a compact set.

That is, $s$-harmonic functions are dense in $C^{k}_{\rm{loc}}$.
This result is clearly in contrast with the rigidity
of harmonic functions in the classical case and can be viewed
as a purely nonlocal feature.
\end{abstract}


\maketitle

\section{Introduction}

It is a well-known fact that harmonic functions are very rigid.
For instance, in dimension~$1$, they reduce to a linear function
and, in any dimension, they never possess local extrema.

The goal of this paper is to show that the situation for
fractional harmonic functions is completely different, namely
one can fix any function in a given domain and find
a $s$-harmonic function arbitrarily close to it.

Heuristically speaking, the reason for this phenomenon is that
while classical harmonic functions are determined
once their trace on the boundary is fixed, in the fractional setting
the operator sees all the data outside the domain. Hence,
a careful choice of these data allows a $s$-harmonic function
to ``bend up and down'' basically without any restriction.

The rigorous statement of this fact is in the following Theorem~\ref{MAIN}.
For this, we recall that, given~$s\in(0,1)$, the fractional Laplace
operator of a function~$u$ is defined (up to a normalizing constant) as
$$ (-\Delta)^s u(x):=\int_{\R^n} \frac{2u(x)-u(x+y)-u(x-y)}{|y|^{n+2s}}\,dy.$$
We refer to~\cites{landkof, stein, silvestre, guida}
for other equivalent definitions, motivations and applications.

\begin{theorem}\label{MAIN}
Fix~${k} \in \N$. Then, given any function~$f\in C^{k}(\overline{B_1})$
and any~$\epsilon>0$, there exist~$R>1$ and~$u\in H^s(\R^n)\cap C^s(\R^n)$
such that
$$ \left\{
\begin{matrix}
(-\Delta)^s u =0 & {\mbox{ in }} B_1,\\
u=0& {\mbox{ in }}\R^n\setminus B_R
\end{matrix}
\right.$$
and
$$ \| f - u \|_{C^{k}(B_1)}\le\epsilon.$$
\end{theorem}

As usual, in Theorem~\ref{MAIN}, we have denoted by~$C^{k}(\overline{B_1})$
the space of all the functions~$f:\overline{B_1}\rightarrow\R$
that possess an extension~$\tilde f\in C^{k}(B_{1+\mu})$
(i.e. $\tilde f=f$ in~$B_1$), for some~$\mu>0$.
\medskip

We also mention that an important rigidity feature
for classical harmonic functions is imposed by Harnack
inequality: namely if~$u$ is harmonic and non-negative in~$B_1$
then $u(x)$ and~$u(y)$ are comparable for any~$x$, $y\in B_{1/2}$.
A striking difference with the nonlocal case is that
this type of Harnack
inequality fails for the fractional Laplacian
(namely it is necessary to require that~$u$ is non-negative
in the whole of~$\R^n$ and not only in~$B_1$, see e.g.
Theorem 2.2 in~\cite{Kas}).
As an application of Theorem~\ref{MAIN}, we point out that one
can construct examples of $s$-harmonic functions
with a ``wild'' behavior,
that oscillate as many times as we want, and reach interior extrema
basically at any prescribed point. In particular,
one can construct $s$-harmonic functions to be used as barriers
basically without any geometric restriction.
\medskip

As a further observation, we would like to stress that, while
Theorem~\ref{MAIN} reveals a purely nonlocal phenomenon,
a similar result does not hold for any nonlocal operator.
For instance, it is not possible to replace ``$s$-harmonic
functions'' with ``nonlocal minimal surfaces'' in the statement
of Theorem~\ref{MAIN}, that is it is not true that any graph
may be locally approximated by nonlocal minimal surfaces.
Indeed, the uniform density estimates satisfied by the nonlocal
minimal surfaces prescribe a severe geometric restriction
that prevent the formation of sharp edges and thin spikes.

We refer to~\cite{CRS} for the definition of nonlocal minimal
surfaces and for their density properties: as a matter
of fact, one of the consequence of Theorem~\ref{MAIN}
is that density properties do not hold true
for $s$-harmonic functions, so $s$-harmonic functions
and nonlocal minimal surfaces may have very different
behaviors.
\medskip

Finally, we would like to point out that, while Theorem~\ref{MAIN}
states that ``up to a small error, all functions are
$s$-harmonic'', it is not true that ``all functions are $s$-harmonic''
(or, more precisely, that any given
function, say in~$B_1$, may be conveniently
extended
outside~$B_1$ to make it $s$-harmonic near the origin).
For instance, any function that vanishes on an open subset of $B_1$
cannot be extended to a function that is $s$-harmonic in $B_1$,
unless it vanishes identically, in view of the Unique Continuation
Principle (see~\cite{FF}). This provides an example of a function
which is not $s$-harmonic in $B_1$ (but, by Theorem~\ref{MAIN}, may
be arbitrarily well approximated by $s$-harmonic functions).
\medskip

We think that it is an interesting problem to determine whether a density 
result as in Theorem~\ref{MAIN}
holds true under additional prescriptions on the function~$u$: for instance, whether one can require also that~$u$ is supported in a ball of universal radius (i.e. independent 
of~$\epsilon$) or whether one can have meaningful bounds on its global norms. 
Moreover, it would be interesting to find constructive and 
efficient algorithms to explicitly determine~$u$.
\medskip

The proof of Theorem~\ref{MAIN} can be summarized in three steps: 
\begin{itemize} 
\item One may reduce to the case in which~$f$ is a 
polynomial, by density in~$C^k(B_1)$, and so to the case in which~$f$ is 
a monomial, by the linearity of the operator. Therefore, it is enough to 
find a $s$-harmonic function that approximates~$x^\beta/\beta!$ 
in~$C^k(B_1)$; 
\item One can construct a $s$-harmonic function~$v$ with 
an arbitrarily large number of derivatives prescribed at a given point: 
in particular, one obtains a $s$-harmonic function that has the same 
derivatives as~$x^\beta/\beta!$ up to order~$|\beta|$ at the origin 
(this is indeed the main step needed for the proof); 
\item One can 
rescale the function~$v$ above by preserving the derivatives of 
order~$|\beta|$ at the origin.
By this rescaling, the higher 
order derivatives (i.e., the derivatives
of order between~$|\beta|+ 1$ and~$k$) go to zero and so they
become a better and better approximation of the higher derivatives
of~$x^\beta/\beta!$, which establishes Theorem~\ref{MAIN}.
\end{itemize}

\medskip

The rest of the paper is organized as follows: in Section~\ref{PO}
we collect some preliminary results, such as a (probably well-known) 
generalization of the Stone-Weierstrass Theorem and the
construction of a $s$-harmonic function in~$B_1$
that has a well-defined growth from the boundary.
Then, in Section~\ref{PP},
we construct a $s$-harmonic function with an arbitrarily large number
of derivatives prescribed.
This is, in a sense, already the core of our argument,
since these types of properties are typical for
the fractional case and do not hold for classical harmonic functions.
Also, from this result, the proof of Theorem~\ref{MAIN}
will follow via a scaling and approximation method.  

\section{Preliminary observations}\label{PO}

In this section we collect some auxiliary
results that will be needed in the rest of the paper.

First of all, 
we recall a version of the Stone-Weierstrass Theorem
for smooth functions. We give a quick proof of it
since in general this result is presented only
in the continuous setting.

\begin{lemma}\label{SW}
For any~$f\in C^{k}(\overline{B_1})$ and any~$\epsilon>0$
there exists a polynomial~$P$ such that~$\|f - P\|_{C^{k}(B_1)}\le\epsilon$.
\end{lemma}

\begin{proof} 
Without loss of generality we may suppose that~$f\in C^{k}_0(B_2)$.
Also, given~$\epsilon>0$ as in the statement of Lemma~\ref{SW},
we fix~$R>0$ such that
\begin{equation}\label{R}\int_{\R^n\setminus B_R} e^{-|x|^2}\,dx\le \epsilon.\end{equation}
Then, we fix~$\eta>0$, to be taken arbitrarily small (possibly
in dependence of~$\epsilon$ and~$R$, which are fixed once and for all),
and we take~$J_\eta\in\N$ large enough such that
\begin{equation}\label{K}
\sum_{j>J_\eta}\frac{(-1)^j}{j!\,\eta^j}\le e^{-1/\eta}.\end{equation}
Let also
\begin{eqnarray*}
&& Q(x) := (\pi\eta)^{-n/2}
\sum_{j=0}^{J_\eta}\frac{(-1)^j\,|x|^{2j}}{j!\,\eta^j}, \\
&& P(x):=\int_{\R^n} f(y) \,Q(x-y)\,dy,\\
{\mbox{and }}&& G(x):=(\pi\eta)^{-n/2} e^{-|x|^2/\eta}.
\end{eqnarray*}
We remark that~$Q$ is a polynomial in~$x$, hence so is~$P$.
Moreover, by a Taylor expansion,
$$ G(x)=Q(x)+(\pi\eta)^{-n/2}
\sum_{j>J_\eta}\frac{(-1)^j\,|x|^{2j}}{j!\,\eta^j} $$
and so, using~\eqref{K}, we conclude that, for any~$x\in B_3$,
\begin{equation}\label{KK}
|G(x)-Q(x)|\le e^{-1/\sqrt\eta},\end{equation}
provided that~$\eta$ is sufficiently small.

Now we recall~\eqref{R}
and we observe that, for any~$\alpha\in\N^n$ with~$|\alpha|\le{k}$
and any~$x\in B_1$,
\begin{eqnarray*}
&&|D^\alpha (G*f)(x)-D^\alpha f(x)|\\&=&
\left| \int_{\R^n} G(y)\Big( D^\alpha f(x-y)-D^\alpha f(x)\Big)\,dy\right|
\\ &\le& \pi^{-n/2}\,\int_{\R^n} e^{-|z|^2}\Big|D^\alpha f(x-\sqrt\eta \,z)-
D^\alpha f(x)\Big|\,dz \\
&\le& 2\pi^{-n/2}\,\epsilon\,\|f\|_{C^{k}(\R^n)}+
\pi^{-n/2}\,\int_{B_R} e^{-|z|^2}\Big|D^\alpha f(x-\sqrt\eta \,z)-
D^\alpha f(x)\Big|\,dz\\ &\le& C\left( 
\epsilon+R^n \sup_{z\in B_R}
\Big|D^\alpha f(x-\sqrt\eta \,z)-
D^\alpha f(x)\Big| \right),
\end{eqnarray*}
for some~$C>0$.
Now, if~$\eta$ is sufficiently small, we have that
$$ \sup_{|x-y|\le \sqrt\eta\, R} \Big|D^\alpha f(x)-
D^\alpha f(y)\Big|\le R^{-n}\epsilon,$$
thus we conclude that
\begin{equation}\label{C-1}
|D^\alpha (G*f)(x)-D^\alpha f(x)|\le C\epsilon,\end{equation}
for any~$\alpha\in\N^n$ with~$|\alpha|\le{k}$
and any~$x\in B_1$,
for a suitable~$C>0$.

Furthermore, using~\eqref{KK} we see that,
for any~$\alpha\in\N^n$ with~$|\alpha|\le{k}$
and any~$x\in B_1$,
\begin{eqnarray*}
|D^\alpha (G*f)(x)-D^\alpha P(x)| &=&
|D^\alpha (G*f)(x)-D^\alpha (Q*f)(x)| \\
&=&\left| \int_{B_3} \Big( G(y)-Q(y)\Big)\,D^\alpha f(x-y)\,dy
\right| \\
&\le& C\,\|f\|_{C^{k}(\R^n)}\,e^{-1/\sqrt\eta}
\\ &\le&\epsilon,\end{eqnarray*}
as long as~$\eta$ is small enough. {F}rom this and~\eqref{C-1} we obtain
$$ \|f-P\|_{C^{k}(\R^n)}\le \|f-(G*f)\|_{C^{k}(\R^n)}+
\|(G*f)-P\|_{C^{k}(\R^n)}\le C\epsilon,$$
for some~$C>0$, which is the desired result, up to
renaming~$\epsilon$.
\end{proof}

Now, we construct a $s$-harmonic function in~$B_1$
that has a well-defined growth from the boundary:
 
\begin{lemma}\label{L1}
Let~$\bar\psi\in C^\infty(\R,[0,1])$ such that~$\bar\psi(t)=0$
for any~$t\in\R\setminus (2,3)$ and~$\bar\psi(t)>0$
for any~$t\in (2,3)$. 

Let~$\psi_0(x):=\bar\psi(|x|)$ and~$\psi\in H^s(\R^n)\cap C^s(\R^n)$
be the solution
of
$$ \left\{
\begin{matrix}
(-\Delta)^s \psi =0 & {\mbox{ in }} B_1,\\
\psi=\psi_0& {\mbox{ in }}\R^n\setminus B_1.
\end{matrix}
\right.$$
Then, if~$x\in \partial B_{1-\epsilon}$, we have that
\begin{equation}\label{St}
\psi(x)=\kappa \,\epsilon^s +o(\epsilon^s)\end{equation}
as~$\epsilon\rightarrow 0^+$,
for some~$\kappa>0$.
\end{lemma}

\begin{proof} 
We notice that the function~$\psi\in H^s(\R^n)$ may be
constructed by the direct method of the calculus of variations,
and also~$\psi\in C^s(\R^n)$, see e.g.~\cite{ros-serra}.

Also, we use the Poisson Kernel representation
(see e.g.~\cites{landkof, claudia}) to write, for any~$x\in B_1$,
\begin{eqnarray*}
\psi(x) &=& c\,\int_{\R^n\setminus B_1} \frac{\psi_0(y)\,(1-|x|^2)^s}{
(|y|^2-1)^s\, |x-y|^n} \,dy \\
&=& c\,(1-|x|^2)^s \int_2^3 \left[ \int_{S^{n-1}} \frac{
\rho^{n-1}\bar\psi(\rho)}{ 
(\rho^2-1)^s\, |x-\rho\omega|^n} \,d\omega\right]\,d\rho,
\end{eqnarray*}
for some~$c>0$.
Now we take~$x\in B_1$, with~$|x|=1-\epsilon$, and we obtain
\begin{eqnarray*}
\psi(x)&=& c \,(2\epsilon-\epsilon^2)^s
\int_2^3 \left[ \int_{S^{n-1}} \frac{
\rho^{n-1}\bar\psi(\rho)}{
(\rho^2-1)^s\, |(1-\epsilon)e_1-\rho\omega|^n} \,d\omega\right]\,d\rho
\\&=&
2^s \,c \,\epsilon^s
\int_2^3 \left[ \int_{S^{n-1}} \frac{
\rho^{n-1}\bar\psi(\rho)}{
(\rho^2-1)^s\, |e_1-\rho\omega|^n} \,d\omega\right]\,d\rho+o(\epsilon^s)
\\ &=& \kappa \,\epsilon^s +o(\epsilon^s),\end{eqnarray*}
for some~$\kappa>0$, as desired.
\end{proof}

We observe that alternative proofs of Lemma~\ref{L1}
may be obtained from a boundary Harnack inequality
in the extended problem and from explicit barriers, see~\cites{CS, ros-serra}.
\medskip

By blowing up the functions constructed in Lemma~\ref{L1}
we obtain the existence of a sequence of~$s$-harmonic
functions approaching~$(x\cdot e)_+^s$, for a fixed unit
vector~$e$, as stated below:

\begin{corollary}\label{C1}
Fixed~$e\in \partial B_1$,
there exists a sequence of functions~$v_{e,j}\in
H^s(\R^n)\cap C^s(\R^n)$ such that~$(-\Delta)^s v_{e,j}=0$ in~$B_1(e)$,
$v_{e,j}=0$ in~$\R^n\setminus B_{4j}(e)$, and
$$ v_{e,j}(x) \rightarrow \kappa (x\cdot e)_+^s \
{\mbox{ in }} L^1(B_1(e)),$$
as~$j\rightarrow+\infty$, for some~$\kappa>0$.
\end{corollary}

\begin{proof} Let~$\psi$ be as in Lemma~\ref{L1} and
$$ v_{e,j}(x):= j^s \psi(j^{-1}x-e).$$
The $s$-harmonicity of~$v_{e,j}$ and the property
of its support can be derived from the ones of~$\psi$.
We now prove the convergence. For this, given~$x\in B_1(e)$
we write~$p_j:= j^{-1}x-e$ and~$\epsilon_j := 1-|p_j|=1-|j^{-1}x-e|$.
We remark that
$$ 1>|x-e|^2 =|x|^2-2x\cdot e+1,$$
which implies that
\begin{equation}\label{9}
{\mbox{$|x|^2<2x\cdot e$, and
$x\cdot e>0$ for all~$x\in B_1(e)$.}}\end{equation} As a consequence
$$ |p_j|^2=|j^{-1}x-e|^2 = j^{-2}|x|^2+1-2j^{-1} x\cdot e=
1-2j^{-1}(x\cdot e)_+ + o(j^{-1})\,(x\cdot e)_+^2$$
and so
$$ \epsilon_j = j^{-1}\,(1+o(1))\,(x\cdot e)_+ .$$
Therefore, using~\eqref{St}, we have
\begin{eqnarray*}
v_{e,j}(x)&=&j^s \psi(p_j)\\&=&
j^s \big(\kappa \epsilon_j^s +o(\epsilon_j^s)\big)
\\ &=& j^s \big(\kappa j^{-s}(x\cdot e)_+^s +o(j^{-s})\big)
\\ &=& \kappa \,(x\cdot e)_+^s +o(1).\end{eqnarray*}
Integrating over~$B_1(e)$ we obtain the desired convergence.
\end{proof}

\section{Spanning the derivative of a function and
proof of Theorem~\ref{MAIN}}\label{PP}

The main result of this section is that we can
find a $s$-harmonic function with an arbitrarily large number
of derivatives prescribed. For this, we use
the standard norm notation for a given multiindex~$\alpha=(\alpha_1,\dots,
\alpha_n)\in\N^n$, according to which
$$ |\alpha|:= \alpha_1+\dots+\alpha_n.$$

\begin{theorem}\label{T1}
For any~$\beta\in\N^n$ there exist~$R> r>0$, $p\in \R^n$,
$v\in H^s(\R^n)\cap C^s(\R^n)$ such that
\begin{eqnarray}
\label{La1}
&& \left\{
\begin{matrix}
(-\Delta)^s v=0 & {\mbox{ in }} B_r(p),\\
v=0& {\mbox{ in }}\R^n\setminus B_R(p),
\end{matrix}
\right. \\
&& \label{La2} {\mbox{$D^\alpha v(p)=0$ 
for any $\alpha\in\N^n$ with~$|\alpha|\le |\beta|-1$,}}\\
&& \label{La3} {\mbox{$D^\alpha v(p)=0$
for any $\alpha\in\N^n$ with~$|\alpha|=|\beta|$ and $\alpha\ne \beta$,}}\\
&& \label{La4} {\mbox{and $D^\beta v(p)=1$.}}
\end{eqnarray}
\end{theorem}

\begin{proof} We denote by~${\mathcal{Z}}$ the set containing
the couples~$(v,x)$
of all functions~$v\in H^s(\R^n)\cap C^s(\R^n)$ and points~$x\in B_r(p)$
that satisfy~\eqref{La1} for some~$R>r>0$ and~$p\in \R^n$.

We let
$$ N:=\sum_{j=0}^{|\beta|} n^j.$$
To any~$(v,x)\in {\mathcal{Z}}$ we can associate a vector in~$\R^N$
by listing all the derivatives of~$v$ up to order~$|\beta|$
evaluated at~$x$, that is
$$ \Big( D^\alpha v(x)\Big)_{|\alpha|\le |\beta|} \in\R^N.$$
We claim that the vector space spanned by this construction
exhausts~$\R^N$ (if we prove this, then
we obtain~\eqref{La2}--\eqref{La4} by writing
the vector with entry~$1$ when~$\alpha=\beta$
and~$0$ otherwise as linear combination of the above functions).

Thus we argue by contradiction, assuming
that the vector space above does not exhaust~$\R^N$
but lies in a subspace. That is, there exists~$c=(c_\alpha)_{|\alpha|\le|\beta|}\in\R^N\setminus\{0\}$
such that
\begin{equation}\label{X0}
\sum_{|\alpha|\le|\beta|}c_\alpha D^\alpha v(x) =0
\end{equation}
for any~$(v,x)\in{\mathcal{Z}}$. 

We observe that the couple~$(v_{e,j},x)$, with $v_{e,j}$
given by~Corollary~\ref{C1} and~$x\in B_1(e)$ belongs to~${\mathcal{Z}}$. Therefore, fixed any~$\xi\in\R^n\setminus B_{1/2}$ and letting~$e:=\xi/|\xi|$,
we have that~\eqref{X0} holds true when~$v:=v_{e,j}$ and~$x\in B_1(e)$. 

Accordingly, for every~$\varphi\in C^\infty_0(B_1(e))$, 
we use integration by parts and the convergence result in Corollary~ \ref{C1}
to obtain that
\begin{eqnarray*}
0 &=& \lim_{j\rightarrow+\infty} \int_{\R^n} \sum_{|\alpha|\le|\beta|}
c_\alpha D^\alpha v_{e,j}(x)
\,\varphi(x)\,dx
\\
&=& \lim_{j\rightarrow+\infty} \int_{\R^n} \sum_{|\alpha|\le|\beta|}
(-1)^{|\alpha|}
c_\alpha \,v_{e,j}(x)
\,D^\alpha \varphi(x)\,dx
\\ &=&
\kappa 
\int_{\R^n} \sum_{|\alpha|\le|\beta|} (-1)^{|\alpha|}
c_\alpha \,
(x\cdot e)_+^s \,D^\alpha \varphi(x)\,dx
\\ &=&
\kappa
\int_{\R^n} 
\sum_{|\alpha|\le|\beta|} c_\alpha \,D^\alpha
(x\cdot e)_+^s \,\varphi(x)\,dx.
\end{eqnarray*}
Consequently, for any~$x\in B_1(e)$,
\begin{equation}\label{10}\sum_{|\alpha|\le|\beta|} c_\alpha \,D^\alpha
(x\cdot e)_+^s =0.\end{equation}
Recalling~\eqref{9}, we observe that, for any~$x\in B_1(e)$,
$$ D^\alpha
(x\cdot e)_+^s = s\,(s-1)\,\dots\,(s-|\alpha|+1)\, (x\cdot e)_+^{s-|\alpha|}
\, e_1^{\alpha_1}\,\dots\,e_n^{\alpha_n}.$$
So we take~$x:=e/|\xi|\in B_1(e)$, and we obtain
$$ D^\alpha
(x\cdot e)_+^s \Big|_{x=e/|\xi|}
= s\,(s-1)\,\dots\,(s-|\alpha|+1)\, |\xi|^{-s}
\, \xi_1^{\alpha_1}\,\dots\,\xi_n^{\alpha_n}.$$
Hence we write~\eqref{10} as
\begin{equation}\label{11}
\sum_{|\alpha|\le|\beta|} c_\alpha
s\,(s-1)\,\dots\,(s-|\alpha|+1) \,\xi^{\alpha}=0,\end{equation}
for any~$\xi\in\R^n\setminus B_{1/2}$.
We remark that equation~\eqref{11}
says that a polynomial in the variable~$\xi$ is identically equal to~$0$
in an open set of~$\R^n$,
therefore all its coefficients must vanish,
namely
\begin{equation}\label{12} s\,(s-1)\,\dots\,(s-|\alpha|+1)
\, c_\alpha=0\end{equation}
for any~$|\alpha|\le|\beta|$.
Notice that none of the terms~$s$, $(s-1)$, $\dots$, $(s-|\alpha|+1)$ vanish
since~$s$ is not an integer. Using this 
we deduce from~\eqref{12} that~$c_\alpha=0$ for any~$|\alpha|\le|\beta|$,
that is~$c=0$, against our assumptions.
\end{proof}

We stress that Theorem~\ref{T1} reflects a purely nonlocal feature.
Indeed, in the local case (i.e. when~$s=1$) the statement of
Theorem~\ref{T1} would be clearly false when~$|m|\ge2$, since the sum
of the pure second derivatives of any harmonic
function must vanish and cannot sum up to~$1$.

With the aid of Theorem~\ref{T1}, we can now complete the proof
of Theorem~\ref{MAIN}:

\begin{proof}[Proof of Theorem~\ref{MAIN}]
By Lemma~\ref{SW},
we can reduce ourselves to the case in which~$f$
is a polynomial.
Consequently, the linearity of the fractional Laplace operator
allows us to
reduce to the case in which~$f$ is a monomial, say
$$f(x)=\frac{x^\beta}{\beta!}$$
for some~$\beta\in\N^n$.
Then we take~$v$ as in Theorem~\ref{T1} and we define
$$ u_\eta(x):= \eta^{-|\beta|} v( \eta x+p),$$
with~$\eta\in(0,1/2)$ to be taken conveniently small in the sequel
(in dependence of~$\epsilon$ that is fixed in the statement of
Theorem~\ref{MAIN}).

The function~$u_\eta$ will be called, for the sake of shortness, simply~$u$
and it will give, for a suitable choice of~$\eta$, the function
seeked in the statement of Theorem~\ref{MAIN}.

Let also~$g(x)
:=u(x)-f(x)=u(x)-(\beta!)^{-1} x^\beta$.
By Theorem~\ref{T1} we know that~$D^\alpha g(0)=0$
for any~$\alpha\in\N^n$ with~$|\alpha|\le|\beta|$.
Furthermore, if~$|\alpha|\ge|\beta|+1$,
$$ |D^\alpha g(x)|= \eta^{|\alpha|-|\beta| } |D^\alpha v(\eta x+p)|
\le C_{|\alpha|}\,\eta \|v\|_{C^{|\alpha|}(B_{1/2}(p))},$$
for any~$x\in B_1$, for some~$C_{|\alpha|}>0$.
As a consequence, defining~${k}':={k}+|\beta|+1$
and fixed any~$\gamma\in\N^n$ with~$|\gamma|\le{k}'-1$ and
any~$x\in B_1$,
we obtain by a Taylor expansion that
$$ D^\gamma g(x) = \sum_{|\beta|+1\le|\gamma|+|\alpha|\le {k}'-1}\frac{
D^{\gamma+\alpha} g(0)}{\alpha!}x^\alpha
+\sum_{|\gamma|+|\alpha|={k}'} 
\frac{{k}'}{\alpha!}\int_0^1 (1-t)^{{k}'-1}
D^{\gamma+\alpha} g(tx)\,dt \,x^\alpha $$
and so~$|D^\gamma g(x)|\le C\eta$, with~$C>0$ possibly depending also on~$v$.

Since this is valid for any~$x\in B_1$ we obtain that
$$ \|u-f\|_{C^{{k}}(B_1)}=
\|g\|_{C^{{k}}(B_1)}\le
\| g\|_{C^{{k}'-1}(B_1)}\le C\eta,$$
for some~$C>0$,
which implies the statement of Theorem~\ref{MAIN}
as long as~$\eta\in(0,C^{-1}\epsilon)$.
\end{proof}

\end{document}